\documentclass[12pt,a4paper,leqno]{amsart}
\usepackage[cp1250]{inputenc}

\usepackage{amsmath}
\usepackage{amsthm}
\usepackage{amssymb}

\newcommand{\sP}{$P$}

\newcommand{\calU}{\mathcal{U}}
\newcommand{\calA}{\mathcal{A}}
\newcommand{\calF}{\mathcal{F}}
\newcommand{\calB}{\mathcal{B}}
\newcommand{\calP}{\mathcal{P}}
\newcommand{\calR}{\mathcal{R}}

\newcommand{\cl}{\operatorname{cl}}
\newcommand{\successors}{\operatorname{succ}}

\renewcommand{\epsilon}{\varepsilon}
\renewcommand{\int}{\operatorname{int}}
\renewcommand{\phi}{\varphi}

\newtheorem{thm}{Theorem}
\newtheorem{pro}[thm]{Proposition}
\newtheorem{lem}[thm]{Lemma}
\newtheorem{cor}[thm]{Corollary}

\title{Scattered \sP-spaces of weight $\omega_1$}


\author{Wojciech Bielas}
\address{Institute of Mathematics, University of Silesia in Katowice, Bankowa 14, 40-007 Katowice, Poland}
\email{wojciech.bielas@us.edu.pl}
\email{andrzej.kucharski@us.edu.pl}
\email{szymon.plewik@us.edu.pl}
\author{Andrzej Kucharski}
\author{Szymon Plewik}

\date{\today}

\begin{document}

\maketitle

\begin{abstract}

We examine dimensional types of scattered \sP-spaces of weight $\omega_1$.
  Such spaces can be embedded into $\omega_2$.
There are established  similarities between dimensional types of scattered separable metric spaces and dimensional types of \sP-spaces of weight $\omega_1$ with Cantor--Bendixson rank less than $\omega_1$.
\end{abstract}

\section{Introduction}\label{sec:1}

A topological space is said to be a \sP\textit{-space}, whenever  $G_\delta$ subsets are open.
A topological space is \textit{scattered} (dispersed) if every non-empty subspace of it contains an isolated point.
If $X$ is a topological space and $\alpha$ is an ordinal number, then $X^{(\alpha)}$ denotes the $\alpha$-th derivative of $X$, compare \cite[p. 261]{kur} or \cite[p. 64]{sie}.
If $X$ is a scattered space, then \textit{Cantor--Bendixson rank} of $X$ is the least ordinal $N(X)$ such that the derivative $X^{(N(X))}$ is empty, see \cite[p. 34]{kec}.
Thus, if $X^{(N(X))}=\emptyset$  and $\beta<N(X)$, then $X^{(\beta)}\neq\emptyset$, also if  $X$ is a scattered space of cardinality $\omega_1$, then
$N(X)<\omega_2$.

This paper is a continuation of \cite{bkp}, where we have investigated crowded \sP-spaces of cardinality and weight $\omega_1$.
Here, we examine scattered \sP-spaces of weight $\omega_1$.
Following the idea that some proofs on \sP-spaces are similar to proofs concerning (scattered) metric spaces, compare \cite[Lemma 2.2.]{dow}, the readers can modify our argumentation to obtain results stated in \cite{gil}, and also contained in \cite{ms} and \cite{tel}.

It will be convenient to use the notation from  \cite{eng} and \cite{jech}.
 A scattered \sP-space is assumed to be regular and  of weight $\omega_1$, nevertheless, we  shall repeat these assumptions in the  statements of facts.
 For brevity, we write  $\gamma\in Lim$ instead of  $\gamma<\omega_2$ is an infinite limit ordinal. 
Also, a closed and open set will be called \emph{clopen}.
The sum of a family of $\kappa$ many homeomorphic copies of a space $X$ we denote  $\bigoplus_{\kappa}X$.
Basic facts about sums can be found in  \cite[pp. 74--76]{eng}.
If topological spaces $X$ and $Y$ are homeomorphic, then we  write $X\cong Y$.
Following  \cite{fre},  \cite[p. 130]{sie} or \cite[p. 112]{kur}, if $X$ is homeomorphic to a subspace of $Y$, then we write $X\subset_h Y$.
If  $X\subset_hY$ and $Y\subset_hX$, then we  write $X=_hY$ and  say that $X$ and $Y$ have the same dimensional type.
 
The paper is organised  as follows.
First, we observe that any scattered space  of weight $\omega_1$ has to be of cardinality $\omega_1$ and then we establish a lemma on embeddings of spaces with a point together with a decreasing base consisting of clopen sets, Lemma \ref{lem:2}.
In Section \ref{sec:3}, we are concerned with properties of elementary sets, i.e. clopen sets with the last non-empty derivative of cardinality $1$.
Lemma \ref{lem:4} says that a scattered \sP-space of weight $\omega_1$ can be  represented as the sum of a family of elementary sets.
Theorem \ref{thm:6} generalises a result of B. Knaster and K. Urbanik, see \cite{ku} and  \cite[Theorem 9]{tel}, that each scattered metric space is homeomorphic to a subspace of an  ordinal number with the order topology.
To be more precise, dimensional types of scattered \sP-spaces of weight $\omega_1$ are represented by dimensional types of subspaces of $\omega_2$.
Corollary \ref{cor:7a} states that  any scattered \sP-space of weight $\omega_1$ has a scattered compactification.
The notion of a stable set enables  us to reduce dimensional types of  scattered \sP-spaces with countable Cantor--Bendixson rank to those of finite ranks.
In Section \ref{sec:6}, we examine spaces $J(\alpha)$ for any $\alpha<\omega_2$, in particular,
we have established that the space $J(\alpha)$ is  maximal among elementary sets with Cantor--Bendixson rank not greater than $\alpha$.  
Our main results are contained in Section \ref{sec:7}.
Theorem \ref{thm:20} and Corollary \ref{cor:21} are  counterparts of \cite[Theorem 19]{gil} and \cite[Corollaries 29 and  31]{gil}.
Finally, we add some remarks  concerning \sP-spaces with uncountable Cantor--Bendixson ranks.
 We think that a more detailed description of such spaces requires new tools, therefore it seems to be troublesome.

\section{Preliminaries}\label{sec:2}

One can readily check the following  properties of a  \sP-space, see \cite{bkp}.
A regular \sP-space has a base consisting of clopen subsets, hence it is completely regular, \cite[Proposition 1]{bkp}.
 For a countable family of open covers, there exists an open cover  which refines each member of this family.
If a regular \sP-space is of cardinality $\omega_1$, then any open cover has a refinement consisting of clopen sets,  \cite[Lemma 14]{bkp}, and  also a countable union  of clopen sets is clopen, \cite[Corollary 15]{bkp}. 

Note that, there exist \sP-spaces of cardinality $\omega_1$ and of weight greater than $\omega_1$.
Indeed, let $X=\omega_1+1$ be equipped with the topology such that countable ordinal numbers are isolated points and
$$\mathcal{B}(\omega_1)=\{U\cup\{p\}\colon U\mbox{ is a club}\}$$
is a base at the point $\omega_1\in X$, where  a closed unbounded subset of $\omega_1$ is called a club, compare \cite[Definition 8.1.]{jech}.
The intersection of countably many clubs is a club and any base for filter generated by the family of all clubs is of cardinality greater than $\omega_1$, which follows from \cite[Lemma 8.4.]{jech}.
Therefore $X$ is a \sP-space of cardinality $\omega_1$ and of weight greater than $\omega_1$.

\begin{pro}\label{pro:1}
A scattered space of weight at most $\omega_1$ is of cardinality at most $\omega_1$.
\end{pro}

\begin{proof}
If $X$ is a scattered \sP-space of weight at most $\omega_1$, then
 $$X=\bigcup\{X^{(\alpha)}\setminus X^{(\alpha+1)}\colon \alpha<\beta\},$$ where   $\beta<\omega_2$.
  The inherited topology of each $X^{(\alpha)}\setminus X^{(\alpha+1)}$ is discrete and of cardinality at most $\omega_1$, hence  $|X|\leq\omega_1$.
\end{proof}

Suppose   $f\colon \omega_1\to \omega_1$ is an injection.
Clearly, we have the following.
$$
\begin{array}{ll}
  \mbox{$(*)$ }&\forall_{\beta<\omega_1}\exists_{\alpha<\omega_1} f[(\alpha,\omega_1)]\subseteq (\beta,\omega_1).\\
\mbox{$(**)$ }&\forall_{\alpha<\omega_1}\exists_{\beta<\omega_1} f^{-1}[(\beta,\omega_1)]\subseteq (\alpha,\omega_1).
\end{array}
$$
The next lemma looks to be known, but for the readers convenience, we present it  with a proof.

\begin{lem}\label{lem:2}
Let $X$ and $Y$ be topological spaces such that  the families of clopen subsets $$\mathcal{B}_x=\{V^x_\alpha\colon\alpha<\omega_1\}\mbox{ and }\mathcal{B}_y=\{U^y_\alpha\colon\alpha<\omega_1\}$$
  are decreasing bases at points $x\in X$ and $y\in Y$,  respectively.
If 
  $$X=\{x\}\cup\bigcup\{V^x_\alpha\setminus V^x_{\alpha+1}\colon \alpha<\omega_1\},$$
and $f\colon \omega_1\to\omega_1$ is an injection, and $\{F_\alpha\colon\alpha<\omega_1\}$ is a family of embeddings  $$F_\alpha\colon V^x_\alpha\setminus V^x_{\alpha+1}\to U^y_{f(\alpha)}\setminus U^y_{f(\alpha)+1},\mbox{ for }\alpha<\omega_1,$$ then $X\subset_h Y$.
\end{lem}

\begin{proof}
We obtain an injection  $F\colon X\to F[X]\subseteq Y$, putting
  $$F(t)=
  \begin{cases}
    y,&\mbox{if }t=x;\\
    F_\alpha(t),&\mbox{if }t\in V^x_{\alpha}\setminus V^x_{\alpha+1}.
  \end{cases}$$
  The sets $V^x_\alpha\setminus V^x_{\alpha+1}$ and $U^y_\alpha\setminus U^y_{\alpha+1}$ are clopen, hence $F[X\setminus\{x\}]\subseteq Y$ is a homeomorphic copy of $X\setminus\{x\}$.
It remains to show that the function $F$ is continuous at the point $x\in X$ and $F^{-1}$ is continuous at the point $y\in Y$.

 Fix a set $U^y_\beta\ni F(x)$.
By  $(*)$  there exists   $\alpha<\omega_1$ such that $f[(\alpha,\omega_1)]\subseteq(\beta,\omega_1)$.
Therefore $$F[V^x_{\alpha+1}]=\{y\}\cup\bigcup_{\gamma>\alpha}F[V^x_\gamma\setminus V^x_{\gamma+1}]\subseteq\{y\}\cup\bigcup_{\gamma>\alpha}U^y_{f(\gamma)}\setminus U^y_{f(\gamma)+1}\subseteq U^y_\beta,$$
hence $F$ is continuous at $x\in X$.

Now,  fix $V^x_\alpha\ni x$.
There exists $\beta<\omega_1$ such that
$$(\beta,\omega_1)\cap f[[0,\omega_1)]\subseteq f[(\alpha,\omega_1)],$$ because of $(**)$. 
We have
 $$ \begin{array}{l}
    U^y_{\beta+1}\cap F[X] = \{y\}\cup\bigcup\{F[X]\cap U^y_{\xi}\setminus U^y_{\xi+1}\colon \xi>\beta\}\subseteq\\
\subseteq    \{y\}\cup\bigcup\{F[X]\cap U^y_{f(\gamma)}\setminus U^y_{f(\gamma)+1}\colon \gamma>\alpha\}=\\
=    \{y\}\cup\bigcup_{\gamma>\alpha}F[V^x_{\gamma}\setminus V^x_{\gamma+1}]= F[V^x_\alpha].
\end{array}$$
Therefore $F^{-1}[U^y_{\beta+1}]\subseteq V^x_\alpha$, hence $F^{-1}$ is continuous at $y\in Y$.
\end{proof}

Let $X$ be a \sP-space.
A base $\mathcal{B}_x=\{V_\alpha\colon \alpha<\omega_1\}$ at a point $x\in X$ is called a \sP\textit{-base} whenever
\begin{itemize}
\item[--] $V_0=X$ and sets $V_\alpha$ are clopen,
  \item[--] $V_\beta\supseteq V_\alpha$ for  $\beta<\alpha<\omega_1$,
  \item[--] $V_\alpha=\bigcap\{V_\beta\colon \beta<\alpha\}$  for a limit ordinal number $\alpha<\omega_1$.
\end{itemize}
Moreover, the sets $V_\alpha\setminus V_{\alpha+1}$ will be called \textit{slices}.
Also,  we have
$$X\setminus\{x\}=\bigcup\{V_\alpha\setminus V_{\alpha+1}\colon\alpha<\omega_1\}.$$

Note that if  $X$ is a \sP-space and $x\in X$, then there exists a \sP-base at point $x\in X$.
Indeed, let $\{V_\alpha\colon \alpha<\omega_1\}$ be a base at a point $x\in X$, which consists of clopen sets.
Putting $W_\alpha=\bigcap_{\gamma<\alpha}V_\gamma$, we obtain the family $\{W_\alpha\colon \alpha<\omega_1\}$ which is a desired \sP-base.

For the purpose of Theorem \ref{thm:10}  we will need the following notions and Lemma \ref{lem:3}.
Let $(P,\leq)$ be an ordered set.
An \textit{antichain} in $P$ is a set $A \subseteq P$ such that any two distinct elements $x, y\in A$ are \textit{incomparable}, i.e., neither $x \leq y$ nor $y \leq x$.
A nonempty $C \subseteq P$ is a \textit{chain} in $P$ if $C$ is linearly ordered by $\leq$.
Now, assume that $(P,\leq)$ is a well-ordered set.
 If $1\leq n<\omega$, then let $\preceq$ be the coordinate-wise order on the product $P^{n}$, i.e. $(a_1,\ldots,a_n)\preceq(b_1,\ldots,b_n)$, whenever $a_i\leq b_i$ for $0< i\leq n$.
 The following variant of Bolzano--Weierstrass theorem seems to be known, it can be  deduced  from \cite[Lemma 28]{gil}.

 \begin{lem}\label{lem:3}
 If $(P,\leq)$ is a well-ordered set, then  any infinite subset of $(P^n,\preceq)$ contains an infinite increasing sequence.
 In particular, any antichain  and any decreasing sequence in $(P^n,\preceq)$ should be finite.\qed
 \end{lem}

\section{On elementary sets}\label{sec:3} 
  
  A clopen subset $E$ of a \sP-space is \textit{elementary}, whenever  the derivative $E^{( N(E)-1)}$ is a singleton.
 Clearly, a singleton is an elementary set and if $E$ is an elementary set, then $N(E)$ is not a limit ordinal. 

\begin{lem}\label{lem:4}
If $X$ is a regular scattered \sP-space of weight $\omega_1$, then any open cover of $X$ can be refined by a partition consisting of elementary sets.
\end{lem}

\begin{proof}
Let $\{U_\gamma\colon \gamma<\omega_1\}$ be an open cover of $X$.
If $ N(X)=1$, then $X$ is a discrete space, so there is nothing to do.
Assume that the hypothesis is fulfilled for each scattered \sP-space $Y$ with $N(Y)<\alpha$.
If  $ N(X)=\alpha$ is a limit ordinal number, then the family $\{X\setminus X^{(\gamma)}\colon \gamma<\alpha\}$ is an open cover of $X$.
So, there exists a partition $\{V_\gamma\colon \gamma<\omega_1\}$ which refines both covers  $\{U_\gamma\colon\gamma<\omega_1\}$ and $\{X\setminus X^{(\gamma)}\colon\gamma<\alpha\}$.
By the induction hypothesis, we can assume that each $V_\gamma$ is the union of elementary subsets, since $N(V_\gamma)\leq\gamma<\alpha$.
In the case   $ N(X)=\beta+1$, the derivative $X^{(\beta)}$ is a discrete space.
Let $\{V_\gamma\colon \gamma<\omega_1\}$ be a partition of $X$ which refines $\{U_\gamma\colon\gamma<\omega_1\}$ and
such that each $V_\gamma\cap X^{(\beta)}$ is a singleton or $V_\gamma\cap X^{(\beta)}=\emptyset$.
If $V_\gamma\cap X^{(\beta)}$ is a singleton, then $V_\gamma$
is an elementary subset.
But if a set $V_\gamma\subseteq X\setminus X^{(\beta)}$, then, by the induction hypothesis, it is the union of a family  of elementary subsets.
\end{proof}

 \begin{pro}\label{pro:5}
   If $X$ is an elementary set and  $\alpha<N(X)$, then there exists an elementary subset $E\subseteq X$ such that $N(E)=\alpha+1$.
   Moreover, if $\alpha+1<N(X)$, then there exists uncountable many pairwise disjoint elementary subsets $E\subseteq X$ such that $N(E)=\alpha+1$.
 \end{pro}

 \begin{proof}
   Assume that $\alpha<N(X)$ and $X^{(\alpha)}\setminus X^{(\alpha+1)}\neq\emptyset$.
   Each point of $X^{(\alpha)}\setminus X^{(\alpha+1)}$ is isolated in $X^{(\alpha)}$, hence there exists $x\in X^{(\alpha)}$ and  a  clopen set $E$ such that $ E\cap X^{(\alpha)}=\{x\}$.
   Clearly, $E$ is an elementary set and $N(E)=\alpha+1$.
   
If $\alpha+1<N(X)$, then fix an elementary set $E\subseteq X\setminus X^{(\alpha+2)}$ with a point  $x\in X^{(\alpha+1)}\cap E$.
We have $N(E\setminus\{x\})=\alpha+1$.
By Lemma \ref{lem:4}, the set $E\setminus \{x\}$  contains an uncountable family of pairwise disjoint elementary subsets, each one with the Cantor--Bendixson rank $\alpha+1$.
 \end{proof}

B. Knaster and K. Urbanik  showed  that a scattered separable metric  space can be embedded in a sufficiently large countable  ordinal number, see \cite{ku}.
Later, R. Telg\'arsky removed the assumption of separability, showing that each metrizable scattered space can be embedded in a sufficiently large  ordinal number, see \cite{tel}.

\begin{thm}\label{thm:6}
Any regular scattered  $P$-space of weight $\omega_1$ can be embedded into $\omega_2$.
\end{thm}

\begin{proof}
We proceed inductively with respect to  the rank $N(Y)<\omega_2$ of scattered $P$-spaces $Y$.
If $N(Y)=1$, then $Y$ is discrete, hence it is homeomorphic to the family of all non-limit countable ordinals.
First, we present  the second step  of the induction.
 Let $Y$ be a scattered \sP-space with  $N(Y)=2$.
The derived set $Y^{(1)}$ is discrete and closed, so we can choose an open cover $\mathcal{P}$  such that if $V\in\mathcal{P}$, then $V\cap Y^{(1)}$ is a singleton or $V=Y\setminus Y^{(1)}$.
Let  $\mathcal{P}^*$ be a partition which refines  $\mathcal{P}$.
Thus each member of  $\mathcal{P}^*$ has at most one accumulation point and also $|\mathcal{P}^*|\leq \omega_1$.
Members of  $\mathcal{P}^*$ should be  homeomorphic to $J(2)$, $i(2)$, $i(2)\oplus D$, or $D$,  where $D$ is a discrete space and  $J(2)=\successors([0,\omega_1^2))\cup\{\omega_1^2\}$, and $i(2)=\successors([0,\omega_1))\cup\{\omega_1\}$.
Thus, one can embed members of $\calP^*$ into successive disjoint intervals of $\omega_2$.

We inductively assume that if $Z$ is a \sP-space with  $N(Z)<\alpha$, then $Z$ is homeomorphic to a subspace of an initial interval  of $\omega_2$.
 Let $Y$ be a \sP-space such that  $N(Y)=\alpha$ is a limit ordinal, so $\mathcal{P} = \{Y\setminus Y^{(\tau)}\colon \tau<\alpha\}$ is  an open cover of $Y$.
  Let $\mathcal{P}^*$ be a partition which  refines  $\mathcal{P}$.
For each $V\in\mathcal{P}^*$,   we have  $N(V)<\alpha$, so  one can embed members of $\mathcal{P}^*$ into successive disjoint intervals of $\omega_2$.

Let $Y$ be a \sP-space such that  $N(Y)=\gamma+1<\omega_2$.
If  $Y^{(\gamma)}=\{z\}$,  then
 fix a \sP-base $\{W_\mu\colon \mu<\omega_1\}$ at  $z\in Y$.
Clearly,
  $$\mathcal{P} =\{W_{\mu}\setminus W_{\mu+1}\colon \mu<\omega_1\},$$
  is a partition of the subspace $Y\setminus\{z\}$, consisting of clopen sets in $Y$.
  Let $Y_1\subseteq [0,\tau_1)$ be a subspace homeomorphic to $Y\setminus W_1$ and   $Y_\mu\subseteq (\tau_\mu,\tau_{\mu+1})$ be homeomorphic to $W_\mu\setminus W_{\mu+1}$, where $(\tau_\mu,\tau_{\mu+1})$ are successive disjoint intervals of $\omega_2$.
  The subspace
  $$\bigcup\{Y_\mu\colon 0<\mu<\omega_1\}\cup\{\sup_{0<\mu<\omega_1} \tau_\mu\}\subseteq [0,\sup_{0<\mu<\omega_1} \tau_\mu]\subseteq\omega_2$$
  is homeomorphic to $Y$, where  the ordinal number $\sup\{\tau_\mu\colon 0<\mu<\omega_1\}$ is assigned to $z$.
  If $Y^{(\gamma)}$ is not a singleton, then $Y$ is the sum of elementary sets with Cantor--Bendixson rank $\gamma+1$.
As previously, we embed these elementary sets into  successive disjoint intervals of $\omega_2$. 
\end{proof}

Theorem \ref{thm:6} shows that all scattered \sP-spaces of weight $\omega_1$ share  topological properties  of the  generalised ordered spaces, compare \cite{lut}.
We omit a detailed discussion of this kind and   confine ourselves to  a counterpart of the  Knaster--Urbanik result, see \cite{ku}.

\begin{cor}\label{cor:7a}
A regular scattered \sP-space of weight $\omega_1$ has a scattered compactification of cardinality $\omega_1$.
\end{cor}

\begin{proof}
 Any regular scattered \sP-space of weight $\omega_1$ has a homeomorphic copy contained in a initial interval of $\omega_2$, thus the closure of this copy is the desired compactification.
 \end{proof}

 Clearly, among regular \sP-spaces only finite ones are compact, so any compactification of an infinite \sP-space is not a \sP-space.

\section{Stable sets with finite Cantor--Bendixson rank}

Assume that $J(0)$ is the empty set and $J(1)$ is a singleton.
But $J(2)=\successors([0,\omega_1^2))\cup\{\omega_1^2\}$, thus $J(2)$ is a \sP-space with exactly one accumulation point $x$ such that there exists a \sP-base $\{V_\alpha\colon\alpha<\omega_1\}$ at $x\in J(2)$ with all slices $V_\alpha\setminus V_{\alpha+1}$ of cardinality $\omega_1$, being discrete as subspaces.
Assume that the \sP-space $J(n-1)$ is defined, then $J(n)$ is the \sP-space with $J(n)^{(n-1)}=\{x\}$ such that there exists \sP-base $\{V_\alpha\colon\alpha<\omega_1\}$ at the point $x\in J(n)$ with $V_\alpha\setminus V_{\alpha+1}\cong \bigoplus_{\omega_1}J(n-1)$, for each $\alpha<\omega_1$.
Analogously, let $i(0)=J(0)$.
If the \sP-space $i(n-1)$ is defined, then $i(n)$ is the \sP-space with $i(n)^{(n-1)}=\{x\}$ and a \sP-base at $x$ such that slices are  homeomorphic to the sum of $\omega$ many copies of $i(n-1)$.

Adapting the idea from \cite[p. 248]{pw},  we  change the definition of a stable set.
Namely, among the  elementary sets  we shall single out stable sets as follows.
Let $E$ be an elementary set such that $E^{(n)}=\{g\}$, where  $n<\omega$.
Considering $E$  as a \sP-space, we say that $E$ is a \textit{stable} set, whenever there is a \sP-base  at $g\in E$ such that any  two slices are homeomorphic.
A singleton is a stable set.
Let $X$ be a \sP-space  such that $X^{(1)}=\{g\}$.
If there exists a \sP-base at $g\in X$ with countably  infinite  slices, then $X$ is a stable set.
By Lemma \ref{lem:2}, such a space $X$ is unique up to homemorphism, in fact it is $i(2)$.
The space $J(2)$ is a stable set, but the elementary set $i(2)\oplus D$ is not stable, whenever  $D$ is uncountable and discrete.
Obviously, $J(2)$ and $i(2)$ are the only stable  sets in the class of all \sP-spaces with Cantor--Bendixson rank $2$.
This class consists of  spaces which have three different dimensional types: $i(2)$, $i(2)\oplus D$ and $J(2)$; and moreover
$$i(2)\subset_hi(2)\oplus D\subset_h J(2).$$
If $E$ is a stable set and $N(E)=n+1$, then the set $E$ is sometimes called $n$\textit{-stable}.
By Lemma \ref{lem:2}, we have  the following.

\begin{cor}\label{cor:6}
  Suppose that $E$ and $F$ are $n$-stable sets with \sP-bases $\{V_\alpha\colon\alpha<\omega_1\}$ and $\{W_\alpha\colon\alpha<\omega_1\}$, respectively, witnessing stability.
  If  $V_0\setminus V_{1}\cong W_0\setminus W_{1}$, then $E\cong F$.\qed
  \end{cor}

In the class of all elementary sets with Cantor--Bendixson rank $3$ there is countably many elementary sets having different dimensional types.
For example, spaces $i(3)\oplus\bigoplus_n J(2)$  are of different dimensional types, depending on $n$.

 \begin{lem}\label{lem:7}
 For each $n<\omega$, there exist only finitely many non-ho\-meo\-mor\-phic  $n$-stable sets.
Also,   any elementary set with finite Cantor--Bendixson rank is the sum of a family of stable sets.
\end{lem}

\begin{proof}
We proceed by induction on $n$.
If $n=1$ and $E$ is an elementary set with  $N(E)=1$, then  there is nothing to do.
Let $\calF_k$ be a family which consists of  all, up to homeomorphism, $k$-stable sets.
Suppose that  the family  $\calF_k$ is finite, for each $k<n$, and  any elementary set with Cantor--Bendixson rank $<n$ is the sum of a family of stable sets.
Suppose $E$ is an elementary set with $E^{(n)}=\{g\}$.
Fix a \sP-base  $\calB=\{V_\alpha\colon \alpha<\omega_1\}$ at $g\in E$.
By the induction hypothesis, assume that each slice $V_\alpha\setminus V_{\alpha+1}$ is  the sum of a family $\calU_\alpha$  of  elements from $\bigcup\{\calF_k\colon k<n\}$.
Let us  define a subsequence of $\calB$ as follows.

\begin{enumerate}
\item[-] If $Y\in\bigcup\{\calF_m\colon m<n\}$ appears only in  countably many families $\calU_\alpha$, then there exists $\beta_Y<\omega_1$ such that there is no $Y$ in $\calU_\alpha$, where $\beta_Y<\alpha<\omega_1$.
  \item[-] If $Y\in\bigcup\{\calF_m\colon m<n\}$ appears uncountable many times only in  countably many $\calU_\alpha$,  then there exists $\gamma_Y<\omega_1$ such that each $\calU_\alpha$ contains at most countably many copies of $Y$, where $\gamma_Y<\alpha<\omega_1$.
  \item[-] For the rest of $Y\in\bigcup\{\calF_k\colon k<n\}$, put $\gamma_Y=\beta_Y=0$.
  \item[-] Having defined $\beta_Y$ and $\gamma_Y$ for each $Y\in\bigcup\{\calF_k\colon k<n\}$, choose an increasing function $f\colon\omega_1\to\omega_1$  such that
    $$\textstyle f(0)>\max\{\max\{\gamma_Y,\beta_Y\}\colon Y\in\bigcup\{\calF_k\colon k<n\}\}$$ and there exists $\tau\in\{0,\omega,\omega_1\}$ such that $Y$ appears $\tau$-many times in each  slice $V_{f(\alpha)}\setminus V_{f(\alpha+1)}$.
  \end{enumerate}
Clearly, the set $V_{f(0)}$ is $n$-stable.
Since  the family $\bigcup\{\calF_k\colon k<n\}$ is  finite, it follows that  the family of all  $n$-stable sets is finite.
By the induction hypothesis and Lemma \ref{lem:4},  the set $E\setminus V_{f(0)}$ is the sum of a family of stable  sets.
  \end{proof}

For technical reasons, the sum  $\bigoplus_0X$ is understood as  the empty set.
  
  \begin{lem}\label{lem:10}
  If $X$ is a scattered \sP-spaces with finite Cantor--Bendixson rank, then  there exists a partition $$\textstyle X\cong \bigoplus_{\kappa_1 }F_1\oplus\ldots\oplus\bigoplus_{\kappa_m }F_m,$$
   where $F_1,\ldots,F_m$ are all stable sets with Cantor--Bendixson rank not greater than $N(X)$ and $\kappa_i \in\omega\cup\{\omega,\omega_1\}$.
  \end{lem}

  \begin{proof}
  Let $X$ be a scattered \sP-space with $N(X)=n$.
  By Lemma \ref{lem:4}, the space $X$ is the sum of a family $\calF$ of elementary sets.
  By Lemma \ref{lem:7}, each $E\in\calF$ is the sum of a family of stable sets.
  Thus, $X$ is the sum of a family of $k$-stable sets, where $k<n$.
Therefore  (again by Lemma \ref{lem:7}), if $F_1,\ldots,F_m$ is a sequence of all (up to homeomorphism) $k$-stable sets, where $k<n$, then 
  $$\textstyle X\cong \bigoplus_{\kappa_1 }F_1\oplus\ldots\oplus\bigoplus_{\kappa_m }F_m,$$
   where  and $\kappa_i \in\omega\cup\{\omega,\omega_1\}$.
  \end{proof}
  
  \begin{thm}
  There are at most countably many non-homeomorphic scattered \sP-spaces with finite Cantor--Bendixson rank.
  \end{thm}

  \begin{proof}
    If $X$ is a scattered \sP-space with $N(X)<n$, then $X$ is homeomorphic to the sum as in Lemma \ref{lem:10}, where $N(F_i)<n$ for each $F_i$.
    There are at most finitely many $k$-stable sets with $k<n$, hence there exist at most countably many such sums determined by the number of occurrences of $k$-stable sets with $k<n$, which suffices to finish the proof.
  \end{proof}
  \begin{cor}\label{cor:12}
  There exist countably many dimensional types of scattered \sP-spaces with finite Cantor--Bendixson rank.\qed
  \end{cor}

  \section{Dimensional type of \sP-spaces with finite Cantor--Bendixson rank}
 Some technical problems of \sP-spaces with countable Cantor--Bendixson rank can be reduced to studying \sP-spaces with
 finite Cantor--Bendixson rank.

\begin{pro}\label{pro:8}
  If $Y$ is an elementary \sP-space of weight $\omega_1$ with $N(Y)=n$, then $Y\subset_h J(n)$.
\end{pro}

\begin{proof}
If $n<2$, then there is nothing to do.
 If $N(Y)=2$, then check that $Y\subset_h J(2)$,  using Lemma \ref{lem:2}.
  Suppose  that the hypothesis is fulfilled for each $k<n $.
Fix a \sP-base $\{W_\gamma\colon\gamma<\omega_1\}$ at the point $y\in Y^{(n-1)}$ and fix a \sP-base $\{V_\gamma\colon\gamma<\omega_1\}$   at the point $x\in J(n)^{(n-1 )}$ such that  $V_\gamma\setminus V_{\gamma+1}=\bigoplus_{\omega_1}J(n-1)$ for each $\gamma<\omega_1$.
  By Lemma \ref{lem:4},
  $$W_\gamma\setminus W_{\gamma+1}\cong \bigoplus\{E_\mu\colon\mu<\omega_1\},$$
  where subsets $E_\mu\subseteq Y$ are  elementary  and with $N(E_\mu)\leq n-1 $.
So, by the induction hypothesis, there exist embeddings
  $$f|_{W_\gamma \setminus W_{\gamma +1}}\colon W_\gamma \setminus W_{\gamma +1}\to V_\gamma \setminus V_{\gamma +1}.$$
Putting $f(y)=x$, we are done.
\end{proof}

\begin{pro}\label{pro:9}
If $X$ is a scattered \sP-space of weight $\omega_1$ such that  $X^{(2n)}=\{g\}$, then $J(n+1)\subset_hX$.
\end{pro}

\begin{proof}
If $n=1$, then $J(2)\cong (X\setminus X^{(1)})\cup\{g\}.$
  Assume that the hypothesis is fulfilled for any scattered \sP-space $Y$ such that $Y^{(2n-2)}=\{y\}$.
  Let
  $$X^*=\{g\}\cup(X\setminus X^{(2n-1)}).$$
  The point $g\in X^*$ has a \sP-base $\{V_\alpha\colon \alpha<\omega_1\}$ such that each slice $V_\alpha\setminus V_{\alpha+1}$ contains $\omega_1$ many pairwise disjoint elementary sets, each with Cantor--Bendixson rank $2n-1$.
  By the induction hypothesis, any such elementary set contains a copy of  $J(n)$, hence $J(n+1)\subset_hX^*$.
\end{proof}

 If $k\geq 2n$ and  $X^{(k)}=\{g\}$, then there exists $Y\subseteq X$ such that the derivative $Y^{(2n)}$ is a singleton, hence $J(n+1)\subset_hY\subseteq X$.
 The assumption $k=2n$ is minimal.
 Indeed, if $n=2$ and $k=3$, then $i(4)$ and $J(3)$ have incomparable dimensional types.

  \begin{thm}\label{thm:10}
    Let $(\calF,\subset_h)$ be an ordered set, where $\calF$ is a family of   scattered \sP-spaces of weight $\omega_1$ with Cantor--Bendixson ranks $\leq n$.
    Then every antichain is finite and every strictly decreasing chain   is finite.
 \end{thm}

 \begin{proof}
   Let   $X$ be a scattered \sP-space with $N(X)=n$  and $F_1,\ldots,F_m$ be all $k$-stable sets with Cantor--Bendixson rank $k\leq n$.
Applying  Lemma~\ref{lem:10}, fix a partition
   $$\textstyle X=_h\bigoplus_{\kappa_1^X}F_1\oplus\ldots\oplus\bigoplus_{\kappa_m^X}F_m,$$
   where $\kappa_i^X\in\omega\cup\{\omega,\omega_1\}=A$.
Thus we have defined a function $\phi$ such that
$$X\mapsto\phi(X)=(\kappa_1^X,\ldots,\kappa_m^X)\in A^m.$$
Consider the coordinate-wise order $(A^m,\preceq)$. 
Using elementary properties of the sum of spaces, we  have the following implications:
\begin{itemize}
\item[(1)] $\phi(X)\preceq\phi(Y)\Rightarrow X\subset_hY$;
\item[(2)] $X\subset_hY\mbox{ and }X\neq_hY\Rightarrow \exists_i\kappa^X_i<\kappa^Y_i$.
\end{itemize}

Condition (1) implies  that if $\calU\subseteq\calF$ is an antichain with respect to $\subset_h$, then $\{\phi(X)\colon X\in\calU\}$ is an antichain in $(A^m,\preceq)$.
Therefore, by Lemma \ref{lem:3}, there is no infinite antichain in $\calF$.

Suppose that  $(X_n)$ is a strictly decreasing sequence with respect to $\subset_h$.
Put $$B_i=\Big\{\{k,n\}\colon \kappa^{X_{\max\{k,n\}}}_{i}<\kappa^{X_{\min\{k,n\}}}_{i}\Big\}.$$
We have  $[\omega]^2=B_1\cup \ldots\cup B_m$, by condition (2).
According to Ramsey's theorem \cite[Theorem A]{ram},  there exist an infinite subset $N\subseteq\omega$ and $i$ such that $\kappa^{X_{n}}_{i}<\kappa^{X_k}_{i}$ for each $k,n\in N$, where $k<n$.
This contradicts $A$ being well-ordered.
 \end{proof}

 Observe that our proof of Theorem \ref{thm:10} uncovers hidden usage  of Ramsey's theorem in  the proof \cite[Lemma 30]{gil}.

 \begin{cor}\label{cor:11}
  Let $(\calF,\subset_h)$ be an ordered set, where $\calF$ is a family of   scattered \sP-spaces of weight $\omega_1$ with finite Cantor--Bendixson rank.
  Then every antichain is finite and every strictly decreasing chain   is finite.
  However, among spaces of $\calF$, there is $\omega$-many but not $\omega_1$-many different dimensional types.
 \end{cor}
 
 \begin{proof}
   Let $\calA$ be an antichain of scattered \sP-spaces with finite Cantor--Bendixson rank.
   If $X,Y\in\calA$, then  $2N(X)<N(Y)$ is impossible.
   Suppose otherwise and put $n=N(X)+1$, then $X\subset_h J(n)$ by Proposition \ref{pro:8}.
   The inequality $N(Y)\geq 2n-1$ and  Proposition \ref{pro:9} imply $J(n)\subset_h Y$, a contradiction.
   Thus $\calA$ has to be finite.

   Suppose $X_1\supset_h X_2\supset_h\ldots$ is a strictly decreasing sequence of scattered \sP-spaces with finite Cantor--Bendixson rank.
   Then all spaces $X_n$ have Cantor--Bendixson rank $\leq N(X_1)$.
   By Theorem \ref{thm:10}, the sequence is finite.

   Observe that if $m\neq n$, then spaces $\bigoplus_mJ(2)$ and $\bigoplus_nJ(2)$ have different dimensional types.
   By Lemma \ref{lem:4} and \ref{lem:7}, there is at most countably many different dimensional types among spaces with Cantor--Bendixson rank $n$, hence no family of dimensional types of spaces with finite Cantor--Bendixson rank can be uncountable.
 \end{proof}

 \section{Maximal elementary  sets}\label{sec:6}

 Proposition \ref{pro:8} states that $J(n)$ is maximal with respect to $\subset_h$ in the class of all \sP-spaces with Cantor--Bendixson rank $\leq n$.
We proceed to a definition of maximal \sP-spaces with infinite Cantor--Bendixson ranks.
Namely, let $J(\omega)$ be the sum of the family $\{J(n)\colon n<\omega\}$, i.e.
$$J(\omega)=\bigoplus\{J(n)\colon n<\omega\}.$$
If $\beta<\omega_2$ and the \sP-space $J(\beta)$ is already defined, then let $J(\beta+1)$ be a \sP-space with $J(\beta+1)^{(\beta)}=\{x\}$ such that there exists a \sP-base $\{V_\alpha\colon \alpha<\omega_1\}$ at $x\in J(\beta+1)$ with all slices $V_\alpha\setminus V_{\alpha+1}$   homeomorphic to the sum $\bigoplus_{\omega_1}J(\beta)$.
If $\beta>\omega$ is a limit ordinal and the space $J(\gamma)$ is defined  for each $\gamma<\beta$,  then put
 $$J(\beta)=\bigoplus\{J(\gamma)\colon\gamma<\beta\}.$$
 Now, we establish  some properties of  $J(\beta)$.

\begin{lem}\label{lem:12}
If  $\gamma<\omega_2$, then  $J(\gamma+1)\cong J(\gamma+1)\oplus \bigoplus_{\omega_1}J(\gamma)$.
 \end{lem}

 \begin{proof}
 Let  $\{x\}=J(\gamma+1)^{(\gamma+1)}$ and let $\{V_\alpha\colon\alpha<\omega_1\}$ be a \sP-base at the point $x\in J(\gamma+1)$.
We have
 $$\textstyle V_0\setminus V_1\cong \bigoplus_{\omega_1}J(\gamma)\mbox{ and
 }V_1\cong J(\gamma+1)=V_0,$$
since $V_0\setminus V_1$ is a clopen set,   we are done.
 \end{proof}

\begin{lem}\label{lem:13}
If $\beta\in Lim$, then $J(\beta)\cong\bigoplus_{\omega_1}J(\beta)$. 
\end{lem}

\begin{proof}
Since
$J(\omega)=\bigoplus\{J(n)\colon n\in\omega\}$, using Lemma \ref{lem:12}, we get
$$\textstyle J(\omega)\cong \bigoplus\{J(n)\oplus\bigoplus_{\omega_1}J(n-1)\colon n>0\},$$
therefore $\textstyle J(\omega)\cong \bigoplus_{\omega_1}J(\omega).$

Assume that if $\gamma\in \beta\cap Lim$, then the hypothesis of the lemma is fulfilled.
According to Lemma \ref{lem:12},  we get
\[\textstyle
  J(\beta)= \bigoplus\{J(\gamma+1)\colon \gamma<\beta\}\oplus \bigoplus\{J(\gamma)\colon \gamma\in \beta\cap Lim\}\cong\]
\[\textstyle
  \bigoplus\{J(\gamma+1)\oplus\bigoplus_{\omega_1}J(\gamma)\colon \gamma<\beta\}\oplus \bigoplus\{J(\gamma)\colon \gamma\in\beta\cap Lim\}\cong\]
\[\textstyle
\bigoplus\{J(\gamma+1)\colon \gamma<\beta\}\oplus \bigoplus_{\omega_1}\bigoplus\{J(\gamma)\colon \gamma<\beta\}\oplus \bigoplus\{J(\gamma)\colon \gamma\in\beta\cap Lim\}.\]
Therefore $\textstyle J(\beta)\cong \bigoplus_{\omega_1}\bigoplus\{J(\gamma)\colon \gamma<\beta\}\cong\bigoplus_{\omega_1}J(\beta).$
\end{proof}

In fact, for each $\beta\in Lim$ we have  $$\textstyle J(\beta)\cong\bigoplus\{\bigoplus_{\omega_1}J(\gamma)\colon \gamma<\beta\mbox{ and }\gamma\notin Lim\}.$$

 \section{Dimensional type of \sP-spaces with countable and infinite Cantor--Bendixson rank} \label{sec:7}  

Note that $J(2)$ cannot be embedded as a clopen subset of $i(n)$.
Indeed, if $U\subseteq i(n)$ is a non-discrete clopen subset, then $U$ contains clopen homeomorphic copy of $i(2)$, but $J(2)$ does not contain a clopen homeomorphic copy of $i(2)$.
Consequently no $i(n)$, for $n>1$, can be homeomorphic to  a clopen subset of $J(\omega)$.
Analogously, no $J(n)$, for $n>1$, can be homeomorphic to  a clopen subset of $\bigoplus\{i(n)\colon n<\omega\}$.
So, one can readily check that $J(\omega)$ is not homeomorphic to $\bigoplus\{i(n)\colon n<\omega\}$.
Nevertheless, we have the following.

 \begin{pro}\label{pro:14}
 If   $X$ is a scattered \sP-space of weight $\omega_1$ such that $N(X)=\omega$, then $X=_hJ(\omega)$. 
 \end{pro}

 \begin{proof}
According to Lemma \ref{lem:4}, let $X=\bigcup\{E_\gamma\colon \gamma<\omega_1\}$ be a partition such that each $E_\gamma$ is an elementary set.
For each $n$, fix $\gamma$ such that $N(E_\gamma)>2n$.
By Proposition \ref{pro:9}, we have $J(n)\subset_h E_\gamma$, consequently $J(\omega)\subset_h X$.
The inequality $N(E_\gamma)<\omega$ and Proposition \ref{pro:8} imply  $E_\gamma\subset_hJ(N(E_\gamma))\subseteq J(\omega)$.
Hence $X\subset_hJ(\omega)$, since  $J(\omega)\cong \bigoplus_{\omega_1}J(\omega)$ by Lemma \ref{lem:13}.
 \end{proof}

 Inductively one can check that $J(n+1)^{(1)}=J(n)$ for each $n\in\omega$, but $J(\omega)^{(1)}\cong J(\omega)$ and therefore, by induction, one  readily checks  $J(\alpha)^{(1)}\cong J(\alpha)$ for any infinite $\alpha$.

\begin{pro}\label{pro:15}
  If $Y$ is an elementary set of weight $\omega_1$ with Cantor--Bendixson rank $\alpha+1$, then $Y\subset_h J(\alpha+1)$.
\end{pro}

\begin{proof}
According to Proposition \ref{pro:8}, the thesis is fulfilled for $\alpha<\omega$.
We simply mimic the proof of Proposition \ref{pro:8}, for other ordinal numbers.
Let $Y^{(\alpha)}=\{y\}$ and suppose  that the hypothesis is fulfilled for each $\beta<\alpha$.
  Let $\{W_\gamma\colon\gamma<\omega_1\}$ be a \sP-base at the point $y\in Y$ and let  $\{V_\gamma\colon\gamma<\omega_1\}$ be a \sP-base at the point $x\in J(\alpha+1)^{(\alpha)}$ such that  $V_\gamma\setminus V_{\gamma+1}=\bigoplus_{\omega_1}J(\alpha)$ for each $\gamma<\omega_1$.
  By Lemma \ref{lem:4},
  $$W_\gamma\setminus W_{\gamma+1}\cong \bigoplus\{E_\mu\colon\mu<\omega_1\},$$
  where sets $E_\mu$ are  elementary  with $N(E_\mu)\leq\alpha$,   for each $\mu<\omega_1$.
 Therefore, by the induction hypothesis, there exist embeddings
  $$f|_{W_\gamma\setminus W_{\gamma+1}}\colon W_\gamma\setminus W_{\gamma+1}\to V_\gamma\setminus V_{\gamma+1}.$$
To finish the proof, put $f(y)=x$.
\end{proof}

 \begin{cor}\label{cor:16}
 If $X$ and $Y$ are elementary sets with Cantor--Bendixson rank $\omega+1$, both  of the weight $\omega_1$, then $X=_hY$.
 \end{cor}

 \begin{proof}
 Assume that  $X=J(\omega+1)$.
 We have $J(\omega+1)^{(\omega)}=\{x\}$ and  $Y^{(\omega)}=\{y\}$.  
Let $\{V_\alpha\colon \alpha<\omega_1\}$ be a \sP-base at the point $y\in Y$ such that each slice $V_\alpha\setminus V_{\alpha+1}$ is the sum of a family $\calR_\alpha$ consisting  of elementary subsets and let $\{U_\alpha\colon\alpha<\omega_1\}$ be a \sP-base at the point $x\in J(\omega+1)$ such that each slice $U_\alpha\setminus U_{\alpha+1}$ is homeomorphic to $J(\omega)\cong \bigoplus_{\omega_1}J(\omega)$.
For each  $E\in\calR_\alpha$, we have $N(E)<\omega$, hence the sum of $\calR_\alpha$ can be embedded into $U_\alpha\setminus U_{\alpha+1}$.
Sending the point $y$ to $x$, we get $Y\subset_h J(\omega+1)$. 

To prove that $J(\omega+1)\subset_hY$, let families $\calR_\alpha$ be as above, assuming, without loss of generality,  that for  every $\alpha<\omega_1$ and $n<\omega$ there exists an elementary set $W\in\calR_\alpha$ such that $W^{(2n)}\neq\emptyset$.
By Proposition \ref{pro:9}, we get  an increasing sequence $(\alpha_n)$ such that $J(n)\subset_h V_{\alpha_n}\setminus V_{\alpha_{n+1}}$.
Therefore  $J(\omega)\subset_h V_0\setminus V_\beta$ for some $\beta<\omega_1$.
Repeating this procedure $\omega_1$ many times, we obtain an unbounded subset  $\{\beta_\gamma\colon\gamma<\omega_1\}\subseteq \omega_1$ such that $J(\omega)\subset_h V_{\beta_\gamma}\setminus V_{\beta_{\gamma+1}}$.
Let $W_\gamma\subseteq V_{\beta_\gamma}\setminus V_{\beta_{\gamma+1}}$ be a copy of $J(\omega)\cong\bigoplus_{\omega_1}J(\omega)$.
Hence $J(\omega+1)\cong \bigoplus\{W_\gamma\colon\gamma<\omega_1\}\cup\{g\}\subseteq Y$. 
\end{proof}

\begin{pro}\label{pro:17}
If $\beta\in Lim$, then  any \sP-space $X$ of weight $\omega_1$ such that $N(X)\leq\beta$ is homeomorphic to a subset of $J(\beta)$, i.e. $X\subset_h J(\beta)$.
\end{pro}

\begin{proof}
By Proposition \ref{pro:14}, the hypothesis is fulfilled for  $\beta=\omega$.
Assume that the hypothesis is fulfilled for each limit ordinal number  $\gamma<\beta$.
Let  $X=\bigoplus\{E_\mu\colon \mu<\lambda\}$, where $\lambda\leq\omega_1$ and each $E_\mu$ is an elementary set with $N(E_\mu)<\beta$.
By the induction hypothesis, fix embeddings $f_\mu\colon E_\mu\to J(\gamma_\mu)$, where   $\gamma_\mu<\beta$.
Since $J(\beta)$ has a representation as the sum $\bigoplus\{\bigoplus_{\omega_1}J(\gamma)\colon \gamma<\beta\mbox{ and }\gamma\notin Lim\}$, one readily checks that
$$\textstyle \bigcup\{f_\mu\colon\mu<\lambda\}\colon X\to \bigoplus\{J(\gamma_\mu)\colon\mu<\lambda\}\cong J(\beta)$$
 is a desired embedding.
\end{proof}

\begin{cor}\label{cor:20a}
  If $X$ is a crowded \sP-space of  weight $\omega_1$ and $Y$ is a scattered \sP-space of weight $\omega_1$, then $Y\subset_h X$.
\end{cor}

\begin{proof}
  According to Proposition \ref{pro:17}, it suffices to prove inductively that $J(\alpha)\subset_hX$ for each $\alpha<\omega_2$.
  Recall that, by \cite[Propositions 1 and 2]{bkp}, any open subset of a crowded \sP-space contains $\omega_1$-many clopen pairwise disjoint clopen subsets, i.e. $X=\bigoplus\{X_\alpha\colon\alpha<\omega_1\}$, where each $X_\alpha$ is crowded.
  
  We have $J(1)\subset_hX$, since $J(1)$ is a one-point space.
  Assume that $J(\gamma)\subset_hZ$ for each $\gamma<\alpha$ and any crowded \sP-space $Z$ of weight $\omega_1$.
  If $\alpha\in Lim$, then $J(\alpha)=\bigoplus_{\gamma<\alpha}J(\gamma)$.
By the induction hypothesis, we have $J(\gamma)\subset_hX_\gamma$, therefore $J(\alpha)\subset_hX$.
  If $\alpha=\beta+1$, then fix $x\in X$ and a \sP-base $\{V_\gamma\colon\gamma<\omega_1\}$ at $x\in X$.
  Let $J(\beta+1)^{(\beta)}=\{g\}$ and $\{W_\gamma\colon\gamma<\omega_1\}$ be a \sP-base at $g\in J(\beta+1)$ such that $W_\gamma\setminus W_{\gamma+1}\cong\bigoplus_{\omega_1}J(\beta)$.
  By the induction hypothesis, there exist embeddings $F_\gamma\colon W_\gamma\setminus W_{\gamma+1}\to  V_\gamma\setminus V_{\gamma+1}$.
By  Lemma \ref{lem:2}, we get   $J(\beta+1)\subset_hX$,   putting $F(g)=x$.
\end{proof}

By Proposition \ref{pro:14} and Corollary \ref{cor:16}, \sP-spaces with the Cantor--Bendixson rank $\omega$ or $\omega+1$ have dimensional type of $J(\omega)$ or $J(\omega+1)$, respectively, and similarly for countable limit ordinals.
 
\begin{thm}\label{thm:18}
If $\beta\in Lim\cap\omega_1$, then $J(\beta)=_hX$ for any \sP-space $X$ of weight $\omega_1$ with $N(X)=\beta$.
Moreover, if $n>0$ and  $Z$ is a \sP-space of weight $\omega_1$ with $N(Z)=\beta+2n-1$, then  $J(\beta+n)\subset_hZ$.
\end{thm}

\begin{proof}
If $\beta=\omega$ and $n=1$, then there is nothing to do, as it is observed just before this theorem.
Assume that
\begin{itemize}
\item[$(*)_\beta$] If $\omega\leq\gamma<\beta$, then there exists $\lambda<\beta$ such that $\gamma<\lambda$ and if $E$ is an elementary set with $N(E)=\lambda$, then $J(\gamma)\subset_hE.$
\end{itemize}
Conditions $(*)_\beta$ suffice to show that  if $N(X)=\beta$, then $J(\beta)\subset_hX$.
Indeed, $X$ is the sum of a family of elementary sets $E_\mu$ such that $\beta$ is the supremum of ordinal numbers $N(E_\mu)$, hence if $\gamma<\beta$, then one can choose an elementary set $E_\mu\supset_h J(\gamma)$, for each $\gamma$ a different one.
Having chosen sets $E_\mu$, we get $\bigoplus\{J(\gamma)\colon\gamma<\beta\}\subset_h X$.
Therefore, by virtue of Proposition \ref{pro:17}, we get $J(\beta)=_hX$.

Now, we prove that  if $n>0$ and $N(Z)=\beta+2n-1$, then $J(\beta+n)\subset_h Z$.
If $N(X)=\beta$, then $J(\beta)=_h X$  as it has been proved above.
Using Lemma \ref{lem:13}, one can readily check that $J(\beta+1)\subset_hX$ whenever $N(X)=\beta+1$.
The inductive steps mimic the proof of Proposition \ref{pro:9}.
Namely, assume  that the hypothesis is fulfilled for any \sP-space $Y$ such that $Y^{(\beta+2n-2)}=\{y\}$, i.e. $J(\beta+n)\subset_hY$.
Now fix a \sP-space $Z$ with $Z^{(\beta+2n)}=\{g\}$.
  Let
  $$Z^*=\{g\}\cup(Z\setminus Z^{(\beta+2n-1)}).$$
  The point $g\in Z^*$ has a \sP-base $\{V_\alpha\colon \alpha<\omega_1\}$ such that each slice $V_\alpha\setminus V_{\alpha+1}$ contains $\omega_1$ many pairwise disjoint elementary sets, each with a one-point $(\beta+2n-2)$-derivative.
  By the induction hypothesis, any such elementary set contains a copy of  $J(\beta+n)$, hence $J(\beta+n+1)\subset_hZ^*\subseteq Z$.
  
  Observe that if $(*)_\beta$ is fulfilled, then $(*)_{\beta+\omega}$ is also fulfilled, which finishes the proof. 
\end{proof}

\begin{pro}\label{pro:19}
If $\beta\in Lim\cap \omega_1$ and $X$ is an elementary set with $N(X)=\beta+1$, then $J(\beta+1)=_hX$.
\end{pro}

\begin{proof}
Let $X^{(\beta)}=\{g\}$ and  $\{V_\alpha\colon\alpha<\omega_1\}$ be a \sP-base at $g\in X$.
By Proposition \ref{pro:15}, we have $X\subset_h J(\beta+1)$.
Fix  an uncountable subset $\{\gamma_\alpha\colon \alpha<\omega_1\}\subseteq\omega_1$, which is  enumerated increasingly and such that $N(V_{\gamma_\alpha}\setminus V_{\gamma_{\alpha+1}})=\beta$ for each $\alpha$.
Putting  $W_\alpha=V_{\gamma_\alpha}$ and bearing in mind that $\bigoplus_{\omega_1}J(\beta)\cong J(\beta)$, we have    $$\textstyle\bigoplus_{\omega_1}J(\beta)\subset_h W_\alpha\setminus W_{\alpha+1},$$
by Theorem \ref{thm:18}.
Applying Lemma \ref{lem:2}, we get $J(\beta+1)\subset_hX$.
\end{proof}

The relation $=_h$ is an equivalence, where
$$[X]_h=\{Y\colon Y=_hX\}$$ is the equivalence class of  $X$.
If $\lambda\in Lim\cup\{0\}$, then let  $\calP_\lambda$ be the family of equivalence classes of  \sP-spaces $X$ such that $\lambda<N(X)\leq\lambda+\omega$.
Putting  $[X]_h<_h[Y]_h$, whenever $X\subset_hY$, we obtain the ordered set, which we denote  $(\calP_\lambda,<_h)$.
By Theorem \ref{thm:18}, classes $[J(\lambda+1)]_h$  and $[J(\lambda+\omega)]_h$ are the least element and the greatest element of $\calP_\lambda$, respectively.
If $[X]_h\in\calP_\lambda$, then $X\setminus X^{(\lambda+1)}$ is the sum of a family of elementary sets with Cantor--Bendixson rank $\lambda+1$.
If $[X]_h\in\calP_\lambda$, then $[X^{(\lambda)}]_h\in\calP_0$ and by  Proposition \ref{pro:19} we have
$$\textstyle[X\setminus X^{(\lambda+1)}]_h=[\bigoplus_{\kappa}J(\lambda+1)]_h,$$
where $\kappa=|X^{(\lambda)}\setminus X^{(\lambda+1)}|$, i.e. $\kappa\in\omega\cup\{\omega,\omega_1\}$, but if $N(X)>\lambda+1$, then $\kappa=\omega_1$.
Note that, the class $[J(\omega^2)]_h$ does not  belong to any family $\calP_\lambda$, despite the fact that if $X$ is a \sP-space with $N(X)<\omega^2$, then  $[X]_h<_h[J(\omega^2)]_h$.
A similar statement holds when $\omega^2$ is replaced by $\gamma\in Lim$ such that  $\gamma\neq \lambda+\omega$ for each $\lambda\in Lim$.

The lemma below is probably well known.

\begin{lem}\label{lem:23}
If $f\colon  X\to Y$ is a continuous injection, then  we have  $f[X^{(\alpha)}]\subseteq Y^{(\alpha)}$,  for any ordinal number $\alpha$.
Moreover, if $X\subset_h Y$, then $X^{(\alpha)}\subset_h Y^{(\alpha)}$.
\end{lem}

\begin{proof}
If $f\colon X\to Y$ is a continuous injection and $x\in X^{(1)}$, then we have
$$f(x)\in\cl f[X\setminus\{x\}]=\cl (f[X]\setminus\{f(x)\}),$$
hence $f(x)\in f[X]^{(1)}\subseteq Y^{(1)}$.
By induction on $\alpha$,  assume  that $\beta<\alpha$ implies  $f[X^{(\beta)}]\subseteq Y^{(\beta)}$.
Thus, if $\alpha=\beta+1$, then
$$f[X^{(\beta+1)}]\subseteq f[X^{(\beta)}]^{(1)}\subseteq (Y^{(\beta)})^{(1)}=Y^{(\beta+1)}.$$
If $\alpha$ is a limit ordinal number, then
$$\textstyle f[\bigcap_{\beta<\alpha}X^{(\beta)}]\subseteq \bigcap_{\beta<\alpha}f[X^{(\beta)}]\subseteq \bigcap_{\beta<\alpha}Y^{(\beta)}=Y^{(\alpha)}.$$

If  $f\colon X\to Y$ is an embedding, then $f[X^{(\alpha)}]\subseteq Y^{(\alpha)}$ and also if a set  $U\subseteq X$ is open, then $f[U\cap X^{(\alpha)}]=f[U]\cap f[X^{(\alpha)}]$
is an open subset of $f[X^{(\alpha)}]$.
\end{proof}

It appears that Lemma \ref{lem:23}  can be reversed.

\begin{lem}\label{lem:24}
If  $Z$ is a \sP-space such that $0<N(Z)\leq\omega$ and $\lambda\in Lim$, then there exists a \sP-space $Z^*$ such that  ${Z^*}^{(\lambda)}\cong Z$.
Moreover, $Z^*\setminus {Z^*}^{(\lambda+1)}$ is the sum $\bigoplus_\kappa J(\lambda+1)$, where $\kappa=|Z\setminus Z^{(1)}|$.
\end{lem}

\begin{proof}
Fix a \sP-space $Z$ such that $0<N(Z)\leq\omega$.
Let $Z^*=Z^{(1)}\cup \bigcup\{J_x\colon x\in Z\setminus Z^{(1)}\}$, where $\{J_x\colon x\in Z\setminus Z^{(1)}\}$ are disjoint copies of $J(\lambda+1)$.
Equip the set $Z^*$ with a topology as follows.
Each subset of the form $J_x$ is clopen in $Z^*$ and it is homeomorphic to $J(\lambda+1)$.
If $z\in Z^{(1)}$ and $V$ is a neighbourhood of $z$ in $Z$, then let
$$V^*=V^{(1)}\cup\bigcup\{J_x\colon x\in V\setminus Z^{(1)}\}.$$
Let the family $\calB_z=\{V^*\colon V\mbox{ is a neighbourhood of }z\in Z\}$ constitute a base at the point $z\in Z^*$.
We leave the reader to check that  the space $Z^*$ is as desired.
\end{proof}

\begin{lem}\label{lem:25a}
If $\lambda\in Lim\cap\omega_1$ and $E$ is a \sP-space with $E^{(\lambda+n)}=\{g\}$, then $E$ contains a subspace $Y=_hJ(\lambda+1)$ such that $Y^{(\lambda)}=\{g\}$.
\end{lem}

\begin{proof}
Fix a \sP-space $E$ such that $E^{(\lambda+n)}=\{g\}$.
If $n=0$, then $E=_hJ(\lambda+1)$  by Proposition \ref{pro:19}.
If $n>0$, then fix a \sP-base $\{V_\alpha\colon \alpha<\omega_1\}$ at $g\in E$ such that each slice $N(V_\alpha\setminus V_{\alpha+1})=\lambda+n-1$.
By Theorem \ref{thm:18}, we can choose a subspace $Y_\alpha\subseteq V_\alpha\setminus V_{\alpha+1}$ such that $Y_\alpha=_h J(\lambda)$.
Therefore  $\{g\}\cup \bigcup\{Y_\alpha\colon \alpha<\omega_1\}\cong J(\lambda+1)$.
\end{proof}

\begin{lem}\label{lem:25}
Let $X,Y$ be scattered \sP-spaces and $\lambda\in Lim\cap\omega_1$.
If $X^{(\lambda)}\subset_h Y^{(\lambda)}$, then $X\subset_h Y$.
\end{lem}

\begin{proof}
Fix an embedding $f\colon X^{(\lambda)}\to Y^{(\lambda)}$.
For each $x\in X^{(\lambda)}\setminus X^{(\lambda+1)}$ there exists an elementary $U_x\subseteq X$ such that $U_x\cap X^{(\lambda)}=\{x\}$.
Without loss of generality, we can assume that sets $U_x$ are pairwise disjoint and, by Theorem \ref{thm:18}, each $U_x=_h J(\lambda+1)$.
Similarly, choose a family $\{V_x\subseteq Y\colon x\in X^{(\lambda)}\setminus  X^{(\lambda+1)}\}$ of pairwise disjoint elementary subsets of $Y$ such that $V_x\cap f[X^{(\lambda)}\setminus X^{(\lambda+1)}]=\{f(x)\}$.
For each $x\in X^{(\lambda)}\setminus X^{(\lambda+1)}$ we have $N(V_x)\geq \lambda+1$ and $U_x=_h J(\lambda+1)$.
Thus we can define an embedding $g_x\colon U_x\to V_x$ as in Lemma \ref{lem:25a}.
If $F\colon X\to Y$  is such that $F|_{X^{(\lambda)}}=f$ and $F|_{U_x}=g_x$, for each $x\in X^{(\lambda)}\setminus X^{(\lambda+1)}$, then $F$ is a desired embedding.
\end{proof}

\begin{thm}\label{thm:20}
If $\lambda$ is a  countable limit ordinal, then ordered sets
\mbox{$(\calP_\lambda,<_h)$} and $(\calP_0,<_h)$ are isomorphic.
\end{thm}

\begin{proof}
By Lemma \ref{lem:23}, if   $[X]_h,[Y]_h\in\calP_\lambda$ and $[X]_h=[Y]_h$, then $[X^{(\lambda)}]_h=[Y^{(\lambda)}]_h$.
Thus we can define $\psi\colon\calP_\lambda\to \calP_0$ by $\psi([X]_h)=[X^{(\lambda)}]_h.$
It remains to show that $\psi$ is a desired isomorphism.

By Lemma \ref{lem:24}, the function $\psi$ is a surjection.
Again, by Lemma \ref{lem:23}, we have
$$[X]_h<_h[Y]_h\Rightarrow \psi([X]_h)<_h\psi([Y]_h),$$
whenever $[X]_h,[Y]_h\in\calP_\lambda$.

By Lemma \ref{lem:25},  if $[X]_h,[Y]_h\in\calP_\lambda$ and $X^{(\lambda)}\subset_h Y^{(\lambda)}$, then $X\subset_h Y$, which implies injectivity of $\psi$.
So, the proof is finished.
\end{proof}

\begin{cor}\label{cor:21}
  Let $\calF$ be a family of   scattered \sP-spaces of weight $\omega_1$ with countable Cantor--Bendixson ranks.
In $(\calF,\subset_h)$, any antichain  and any strictly decreasing chain are finite.
\end{cor}

\begin{proof}
  Assume that $\calA$ is an antichain in $(\calF,\subset_h)$ and $X\in\calA$, and $N(X)=\lambda+n$, where $\lambda\in Lim$ and $n<\omega$.
  By Theorem \ref{thm:18} and Proposition \ref{pro:17},   if $Y\in\calA$, then $N(Y)=\lambda+m$, where $m<\omega$ and $\lambda<\omega_1$. 
  By Corollary \ref{cor:11} and Theorem \ref{thm:20}, the family $\calA$ has to be finite.
  Analogously, we proceed in the case $\calA$ is a decreasing chain.
\end{proof}

\section{A few remarks on \sP-spaces with Cantor--Bendixson rank $\geq\omega_1+1$}

As we have learned there is only one sensible way of defining an elementary set with Cantor--Bendixson rank $\beta+1$  for $\beta\in Lim$, since any two such elementary sets have the same dimensional type by Proposition \ref{pro:19}.
This is not the case for elementary sets with Cantor--Bendixson rank $>\omega_1$.
Namely, let $Y(\omega_1)$ be an elementary set such that $Y(\omega_1)^{(\omega_1)}=\{g\}$  and each slice $V_\alpha\setminus V_{\alpha+1}=J(\alpha)$.

 \begin{pro}\label{pro:24}
   If $X$ is  a \sP-space of the weight $\omega_1$ such that $|X^{(\omega_1)}|=1$, then $X=_hY(\omega_1)$ or $X=_hJ(\omega_1+1)$, or $X=_hY(\omega_1)\oplus J(\omega_1)$.
 \end{pro}

 \begin{proof}
Let $X^{(\omega_1)}=\{x\}$.
Let $\{V_\alpha\colon \alpha<\omega_1\}$ be a \sP-base at the point $x\in X$.

If each slice $V_\alpha\setminus V_{\alpha+1}$ has Cantor--Bendixson rank $\omega_1$, then, by Proposition \ref{pro:15}, $X\subset_h J(\omega_1+1)$ and
$$\textstyle V_\alpha\setminus V_{\alpha+1}=_h\bigoplus_{\beta<\omega_1}J(\beta)=J(\omega_1),$$
for each $\alpha<\omega_1$.
Therefore $X=_h J(\omega_1+1)$.

If each slice $V_\alpha\setminus V_{\alpha+1}$ has Cantor--Bendixson rank $<\omega_1$, then, by Proposition \ref{pro:17}, for any  slice $V_\alpha\setminus V_{\alpha+1}$ there exists $\gamma<\omega_1$ such that $V_\alpha\setminus V_{\alpha+1}\subset_h J(\gamma)$.
Therefore,   we have $X\subset_h Y(\omega_1)$, by Lemma \ref{lem:2}.
By Theorem \ref{thm:18},  for any $\gamma<\omega_1$ there exists $\alpha<\omega_1$ such that $J(\gamma)\subset_h V_\alpha\setminus V_{\alpha+1}$, therefore we get $X=_h Y(\omega_1)$.

Without loss of generality, it remains to consider the case when  $N(V_0\setminus V_1)=\omega_1$ and $N(V_\alpha\setminus V_{\alpha+1})<\omega_1$ for $\alpha\geq 1$.
Then we have
$$X=V_1\cup (V_0\setminus V_1)=_hY(\omega_1)\oplus J(\omega_1),$$
which completes the proof.
\end{proof}

One can readily check that
$$J(\omega_1)\subset_h Y(\omega_1)\subset_h Y(\omega_1)\oplus J(\omega_1)\subset_h J(\omega_1+1)$$
and no two of these four spaces have the same dimensional type.

\section{Conclusions}
Compare cardinal characteristics of some classes of dimensional types with classes of non-homeomorphic spaces.

By Corollary \ref{cor:12} and Theorems \ref{thm:18} and \ref{thm:20}, if $\lambda<\omega_1$, then there are only  countably many dimensional types of \sP-spaces with Cantor--Bendixson ranks $\leq\lambda$.
Therefore, there is exactly $\omega_1$-many dimensional types of \sP-spaces with countable Cantor--Bendixson ranks.

 There exist $2^{\omega_1}$-many non-homeomorphic \sP-spaces $X$ with $N(X)=\omega_1$, but if $\lambda\in Lim\cap \omega_1$, then there are continuum many non-homeomorphic \sP-spaces $X$ with $N(X)=\lambda$.
 Indeed, for each
 $$A\subseteq\{\alpha+2k\colon \alpha\in Lim\cap\lambda\mbox{ and } k<\omega\},$$  let $X$ be a  \sP-space such that
$$X^{(\alpha)}\setminus X^{(\alpha+2)}=_h
\begin{cases}
\bigoplus_{\omega_1}i(2),&\mbox{if }\alpha\in A,\\
\bigoplus_{\omega_1}J(2),&\mbox{if }\alpha\notin A;\\
\end{cases}$$
constructing such a \sP-space  needs some extra work which we leave to the reader.
We get that different subsets of $\lambda$ are assigned to non-homeomorphic \sP-spaces.
Hence,  there are continuum many non-homeomorphic scattered \sP-spaces with countable Cantor--Bendixson rank. 
Similarly, one can prove that if $\lambda=\omega_1$, then there exist $2^{\omega_1}$-many non-homeomorphic scattered \sP-spaces.

It seems that examination of \sP-spaces with uncountable Cantor--Bendixson ranks needs several new ideas and extra efforts.
As it has been noted earlier,  the readers would easily conclude  results concerning scattered separable metric spaces mimicking our argumentation.
The same can be said about the so-called $\omega_\mu$-additive spaces, introduced by R. Sikorski \cite{sik}, compare \cite[p. 1]{mis}.
Indeed, consider the family of all scattered $\omega_\mu$-additive spaces of weight $\omega_\mu$.
Replacing $\omega_1$ by $\omega_\mu$, one can adapt  our results to this family.


\begin{thebibliography}{40}
\bibitem{bkp} W. Bielas, A. Kucharski, Sz. Plewik, \textit{Inverse limits which are P-spaces}. Topology Appl. 312 (2022), Paper No. 108064, 18.

  \bibitem{dow} A. Dow, \textit{More topological partition relations on $\beta\omega$}. Topology Appl. 259 (2019), 50--66.

  \bibitem{eng} R. Engelking, \textit{General topology}. Mathematical Monographs, Vol. 60, PWN--Polish Scientific Publishers, Warsaw, (1977).

  \bibitem{fre} M. Fr\'echet, \textit{Les dimensions d'un ensemble abstrait}.
Math. Ann. 68 (1910), no. 2, 145--168.

  \bibitem{gil} W. D. Gillam, \textit{Embeddability properties of countable metric spaces}. Topology Appl. 148 (2005), no. 1--3, 63--82.

  \bibitem{jech} T. Jech, \textit{Set theory}. Springer Monographs in Mathematics. Springer-Verlag, Berlin, (2003).

    \bibitem{kec} A. S. Kechris, \textit{Classical descriptive set theory}. Graduate Texts in Mathematics 156,  Springer-Verlag, New York, (1995).

  \bibitem{ku} B. Knaster, K. Urbanik, \textit{Sur les espaces complets s\'eparables de dimension 0}. Fund. Math. 40 (1953), 194--202.

\bibitem{kur} K. Kuratowski, \textit{Topology}. Vol. I. Academic Press, New York-London; Pa\'nstwowe Wydawnictwo Naukowe, Warsaw (1966).

\bibitem{lut}D. J. Lutzer, \textit{On generalized ordered spaces}. Dissertationes Math. (Rozprawy Mat.) 89 (1971), 32 pp. 

\bibitem{ms} S. Mazurkiewicz, W. Sierpi\'nski,
\emph{Contribution \`{a} la topologie des ensembles d\'enombrables}. Fundamenta Mathematicae 1 (1920), 17--27.

\bibitem{mis} A. K. Misra, \textit{A topological view of P-spaces}. General Topology and Appl. 2 (1972), 349--362.

\bibitem{pw} Sz. Plewik, M. Walczy\'nska, \textit{On metric $\sigma$-discrete spaces}. Algebra, logic and number theory, 239--253, Banach Center Publ., 108, Polish Acad. Sci. Inst. Math., Warsaw, (2016).

  \bibitem{ram}F. P. Ramsey, \textit{On a Problem of Formal Logic}. Proc. London Math. Soc. (2) 30 (1929), no. 4, 264--286.

\bibitem{sie} W. Sierpi\'nski, \textit{General topology}. Translated by C. Cecilia Krieger. Mathematical Expositions, No. 7, University of Toronto Press, Toronto, (1952).

\bibitem{sik}R. Sikorski, \textit{Remarks on some topological spaces of high power}. Fund. Math. 37 (1950), 125--136. 

  \bibitem{tel} R. Telg\'arsky, \textit{Total paracompactness and paracompact dispersed spaces}. Bull. Acad. Polon. Sci. S\'er. Sci. Math. Astronom. Phys. 16 (1968), 567--572. 

\end{thebibliography}
\end{document}